\documentclass[letterpaper, 11 pt, onecolumn]{ieeeconf}  

\IEEEoverridecommandlockouts                              

\usepackage{xcolor}

\usepackage[fleqn]{amsmath}
\usepackage{amssymb,amsfonts,amsthm} 
\usepackage{tabularx}
\usepackage{mathabx}
\usepackage{graphicx,url} 
\usepackage{bm}
\usepackage{float}
\usepackage{makecell}
\let\labelindent\relax
\usepackage{enumitem}
\usepackage{subcaption}
\usepackage{wrapfig}
\usepackage{algorithm,algorithmicx} 
\usepackage{algpseudocode} 
\usepackage{multirow}
\usepackage{diagbox}

\usepackage{cite}
\usepackage{soul}  
\usepackage[bookmarks=true]{hyperref}
\hypersetup{
    colorlinks=true,
    pdfborderstyle={/S/U/W 1},
    linkcolor=blue,
    filecolor=magenta,      
    urlcolor= blue,
    linkbordercolor=red,
}
\newcommand\rurl[1]{%
  \href{http://#1}{\nolinkurl{#1}}%
}

\usepackage{comment}
\usepackage{accents}
\newcommand{\ubar}[1]{\underaccent{\bar}{#1}}

\makeatletter
\def\cl@part {\@elt {chapter}}
\makeatother

\usepackage{cleveref}

\crefname{equation}{}{} 
\crefname{lemma}{Lemma}{Lemmas}
\crefname{theorem}{Theorem}{Theorems}
\crefname{table}{Table}{Tables}
\crefname{figure}{Fig.}{Figs.}
\crefname{remark}{Remark}{Remarks}
\crefname{assumption}{Assumption}{Assumptions}
\crefname{section}{Section}{Sections}
\crefname{definition}{Definition}{Definitions}
\crefname{algorithm}{Algorithm}{Algorithms}
\crefname{proposition}{Proposition}{Propositions}
\crefname{appendix}{Appendix}{Appendices}

\usepackage{accents}

\let\abs\relax
\newcommand{\abs}[1]{\left\lvert#1\right\rvert}
\newcommand{\norm}[1]{\left\lVert#1\right\rVert}

\def\tilx{\tilde{x}}

\def\mbR{\mathbb{R}}
\def\mbRn {\mathbb{R}^n}

\def\mbZ{\mathbb{Z}}
\def\mbz#1{\mathbb{Z}_1^{#1}}

\def\mcF{\mathcal{F}}

\def\mcX{\mathcal{X}}
\def\mcL{\mathcal{L}}

\def\mcU{\mathcal{U}}

\def\mcV{\mathcal{V}}

\def\trieq{\triangleq}

\newtheorem{theorem}{Theorem}
\newtheorem{lemma}{Lemma}
\newtheorem{proposition}{Proposition}

\theoremstyle{definition}  \newtheorem{definition}{Definition}
\theoremstyle{definition} 
\newtheorem{assumption}{Assumption}
\theoremstyle{remark}  
\newtheorem{remark}{Remark}

\makeatletter
\renewcommand*\env@matrix[1][\arraystretch]{%
  \edef\arraystretch{#1}%
  \hskip -\arraycolsep
  \let\@ifnextchar\new@ifnextchar
  \array{*\c@MaxMatrixCols c}}
\makeatother

\def \tilx {\tilde x}

\def\mcS{{\mathcal{S}}}

\def \mcK{{\mathcal{K}}}

\def\sos#1{{\Phi_\textup{sos}(#1)}} 

\def \bzero {\mathbf{0}}

\usepackage{commath}

\def \thetadot {\dot\theta} 

\def \onetoN#1 {1,\dots,#1}

\def \nu {{n_u}}
\def \ntheta {{n_\theta}}
\def \Ntheta {{N_\theta}}

\def \sym#1{\left\langle #1\right\rangle}

\def \fThetaV  {\mcF_\Theta^\mcV}

\def \ver#1{\textup{ver}(#1)}

\usepackage{textcomp}
\makeatletter
\newcommand\fs@spaceruled{\def\@fs@cfont{\bfseries}\let\@fs@capt\floatc@ruled
  \def\@fs@pre{\vspace{3mm}\hrule height.8pt depth0pt \kern2pt}%
  \def\@fs@post{\kern2pt\hrule\relax\vspace{-4mm}}%
  \def\@fs@mid{\kern2pt\hrule\kern2pt}%
  \let\@fs@iftopcapt\iftrue}
\makeatother

\author{Pan Zhao
\thanks{P. Zhao is with the Department of 
Aerospace Engineering and Mechanics, University of Alabama, Tuscaloosa, AL 35487, USA. {\tt\small pan.zhao@ua.edu}. }%
}
\def\BibTeX{{\rm B\kern-.05em{\sc i\kern-.025em b}\kern-.08em
    T\kern-.1667em\lower.7ex\hbox{E}\kern-.125emX}}
\begin{document}

\title{\LARGE \bf Parameter-Dependent Control Lyapunov Functions for Stabilizing Nonlinear Parameter-Varying Systems}

\maketitle
\thispagestyle{empty}
\pagestyle{empty}

\begin{abstract}

This paper introduces the concept of parameter-dependent (PD) control Lyapunov functions (CLFs) for gain-scheduled stabilization of nonlinear parameter-varying (NPV) systems. It shows that given a PD-CLF, a min-norm control law can be constructed by solving a robust quadratic program. For polynomial control-affine NPV  systems, it provides convex conditions, based on the sum of squares (SOS) programming, 
to jointly synthesize a PD-CLF and a PD controller while maximizing the PD region of stabilization. Input constraints can be straightforwardly incorporated into the synthesis procedure.  Unlike traditional linear parameter-varying (LPV) methods that rely on linearization or over-approximation to get an LPV model, the proposed framework fully captures the nonlinearities of the system dynamics. 
The theoretical results are validated through numerical simulations, including a 2D rocket landing case study under varying mass and inertia.

\end{abstract}


\section{Introduction}\label{sec:introduction}
Control Lyapunov functions (CLFs) offer a method for designing stabilizing control inputs for nonlinear systems without the need for an explicit feedback control law \cite{sontag1983clf, Artstein1983clf}. Additionally, a CLF can be used to generate a pointwise min-norm control law, which is optimal for some meaningful cost functions \cite{freeman1996control-clf}. Given a nonlinear system, verification and synthesis of a CLF has been studied using sum of squares (SOS) optimization techniques \cite{tan2004searching-clf-sos,dai2022convex-clf-cbf}. However, existing work on CLF typically considers only autonomous systems or time-varying systems \cite{jiang2009stabilization-ntv-clf}. However, the concept of CLF for nonlinear parameter-varying systems has not been studied. 

A nonlinear parameter-varying (NPV) system is a nonlinear system that explicitly depends on some time-varying parameters. It can be considered as a generalization of a linear parameter-varying (LPV) system \cite{Shamma92gain-CSM,Rugh00gs_survey, Moh12LPVBook} in the sense that the dynamics of an NPV system can have a nonlinear structure. In fact, many nonlinear time-varying systems in the real world can be modeled as an NPV system, such as a rocket with varying mass due to propellant consumption, a ground vehicle experiencing varying friction coefficients, and an aircraft subject to varying disturbances.  Similar to the LPV paradigm \cite{Shamma92gain-CSM,Rugh00gs_survey, Moh12LPVBook},  modeling a nonlinear time-varying system as an NPV system allows us to design gain-scheduled nonlinear controllers that automatically adjust the control law based on real-time measurement and estimation of the parameters. 



{\bf Related Work}: {\it LPV control}: In LPV control framework \cite{Shamma92gain-CSM,Rugh00gs_survey, Moh12LPVBook}, a nonlinear system is first approximated by an LPV model, $\dot x = A(\theta(t))x+ B(\theta(t))u$, whether $\theta(t)$ is the vector of scheduling parameters that can be measured or estimated online 
 and characterize the variation of system dynamics. 
 Given the LPV model, an LPV controller can be designed to ensure stability and performance of the closed-loop system, e.g., using linear matrix inequalities \cite{Wu96Induced,Apk98}. A major issue with the LPV framework is that the LPV model used for control design is either {\it only locally valid} when Jacobian linearization is used to obtain the LPV model, or is an {\it over-approximation} of the real dynamics when the quasi-LPV approach is used, leading to conservative performance \cite{Rugh00gs_survey}. 
 Additionally,  in the LPV framework, the control law is limited to be linear (although it can be parameter-dependent).


{\it Verification and synthesis of CLF and control barrier functions (CBFs)}: Inspired by CLFs, CBFs are proposed to design control inputs to ensure set-invariance of nonlinear systems without resorting to a specific control law \cite{ames2016cbf-tac}. Due to the similarity between a CLF and a CBF, approaches developed for verifying and synthesizing one of them can be readily applied to the other.  \cite{tan2004searching-clf-sos} proposed to iteratively search for CLFs via SOS optimization without searching for a controller. 
SOS optimization based CBF synthesis has been studied in   \cite{wang2018permissive-sos, clark2021verification} without considering input constraints, and in \cite{dai2022convex-clf-cbf,zhao2023convex-cbf} with consideration of input constraints. 
However, all these existing works are for nonlinear autonomous systems, and cannot handle the NPV systems considered in this manuscript. 

{\bf Statement of Contributions}:
This paper introduces the concepts of parameter-dependent CLF (PD-CLF) for stabilizing NPV systems. It shows that given a PD-CLF, a min-norm control law can be constructed by solving a robust quadratic program and is locally Lipschitz. Additionally, it presents a method for
jointly synthesizing a PD-CLF and a PD controller while maximizing the PD region of stabilization based on SOS programming. 
The theoretical results are validated through numerical simulations, including a 2D rocket landing case study under varying mass and inertia.

{{\it Notations}: Let $\mathbb{R}$, $\mathbb{R}^n$,  $\mathbb{R}^{m\times n}$ and $\mcS^n $ denote the sets of real numbers, $n$-dimensional real vectors, $m$ by $n$ real  matrices, and $n\times n$ real symmetric matrices, respectively. $\mbZ_i^n$ denotes the integer set $\{i, i+1,\cdots,n\}$.
$\mbR[x]$ denotes the set of polynomials with real coefficients, while $\mcS^n[x] $ and $\mbR^{m\times n}[x]$ denote the sets of $n\times n$ real symmetric polynomial matrices and of ${m\times n}$ real matrices, respectively, whose entries are polynomials of $x$ with real coefficients. Given a matrix $A$,  $\langle A \rangle$ denotes $A+A^T$. Finally, we let 
 $\sos{x}$ denote the set of SOS polynomials of $x$. 
}

\section{Preliminaries}\label{sec:preliminaries}
Consider a nonlinear parameter-varying (NPV) system 
\begin{equation}\label{eq:dynamics}
    \dot{x} = f(x,\theta)+g(x,\theta)u, 
\end{equation}
where $x(t)\in \mbR^{n}$, $u(t) \in \mbR^m$, $\theta(t)\in\mbR^{n_\theta}$ is the vector of parameters that can be time-varying and measurable online, and $f:\mbR^n\times \mbR^\ntheta\ \rightarrow \mbR^n$ and  $g:\mbR^n\times
\mbR^\ntheta \rightarrow \mbR^{n\times m}$ are locally Lipschitz in $\theta$ and $x$.   As an example, $\theta(t)$ could represent the payload and propeller control effectiveness of a delivery drone, the friction coefficient of a ground vehicle, the propellant mass of a spacecraft or a rocket, or the disturbance forces applied to an aircraft. 
\begin{remark} Unlike existing NPV work \cite{
fu2018hinf-npv,fu2019exponential-npv,lu2020domain-npv,yu2024hinf-npv},  
    we do not require $\dot \theta$ to be online measurable. This makes our solution more practical as in the real world it is often challenging to obtain $\dot \theta$.  
\end{remark}
\begin{assumption} \label{assump:theta-thetadot-range}
The parameters and their derivatives satisfy
\begin{equation}\label{eq:theta-thetadot-constrs}
   \theta(t) \in \Theta \textup{ and } \thetadot(t)\in \mcV,~\forall t\geq 0, 
\end{equation}
 where $\Theta$ is a semi-algebraic set defined by a finite number of polynomials $h_i(\theta)$ as $\Theta\trieq \{\theta\in\mbR^\ntheta: h_i(\theta)\geq 0, i\in \mbz{\Ntheta}\}$,~and $\mcV$ is a hyper-rectangular set defined by $\mcV \trieq \{ v\in\mbR^\ntheta: \ubar v_i\leq  v_i\leq \bar v_i, i\in\mbz{\ntheta}\}$. 
\end{assumption}For ease of exposition, we denote the set of admissible trajectories of $\theta$ satisfying \cref{eq:theta-thetadot-constrs} as
\begin{equation}\label{eq:F-Theta-mcV-defn}
 \mcF_\Theta^\mcV\trieq \left\{\theta:\mbR^+\rightarrow \mbR^\ntheta| \theta(t) \in \Theta, \dot\theta(t)\in \mcV,\ \forall t\geq 0\right\}
\end{equation}
For brevity, we often omit the dependence of variables and functions on $t$ hereafter. 

Without loss of generality, we assume the equilibrium point of \cref{eq:dynamics} is $x=u=\bzero$, which indicates that $f(0,\theta)=0$ for all admissible trajectoreis of $\theta$.

\section{Parameter-Dependent CLFs for Stabilizing NPV Systems}\label{sec:pd-clf}
\begin{definition}
The equilibrium point $x=0$  of system \cref{eq:dynamics} is uniformly asymptotically stabilizable if there exists an admissible controller $u(x,\theta)$ that is locally Lipschitz 
such that the closed loop system $ \dot{x} = f(x,\theta)+g(x,\theta)u(x,\theta)$ is uniformly asymptotically stable (UAS), i.e., there exist a class $\mcK\mcL$ function $\beta$ and a positive constant $c$, independent of $t_0$, such that 
    \begin{equation}
   \hspace{-3mm}    \norm{x(t)} \! \leq \! \beta \left(\norm{x(t_0)}, t\!-\!t_0\right), \   \forall t\geq t_0 \geq 0, \ \forall \norm{x(t_0)}\!<\! c. 
    \end{equation}
\end{definition}

\begin{definition}\label{def:pd-clf}
   A continuously differentiable function $V(x,\theta)$ is a PD-CLF for \cref{eq:dynamics} if there exist continuous positive definite functions $\alpha_1$, $\alpha_2$, and  $\alpha$ such that the following conditions hold:   
\begin{subequations}\label{eq:clf-cond}
    \begin{align}
   &   \alpha_1(\norm{x}) \leq V(x,\theta)  \leq \alpha_2(\norm{x}),\  \forall \theta \in \Theta, 
     \label{eq:clf-cond-V-bnd}   \\
 &  \inf_{u} { L_f V + L_g V u +  \frac{\partial V}{\partial \theta }\thetadot} \leq -\alpha( \norm{x}) ,\ (\theta,\dot \theta)\in \Theta \!\times\! \mcV, \label{eq:clf-cond-derivative} 
       \vspace{-10mm}
    \end{align}
\end{subequations} 
where $L_fV= \frac{\partial V}{\partial x}f$ and $L_gV= L_g V$.
\end{definition}
A geometric illustration of a PD-CLF is given in \cref{fig:pd-clf-illustration}.
\begin{figure}
    \centering
    \includegraphics[width=0.4\columnwidth]{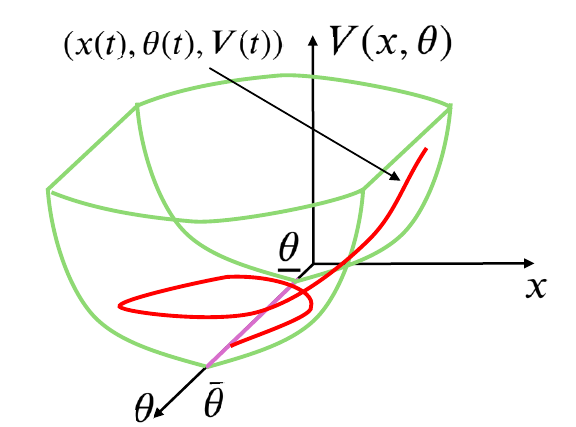}
    \caption{Geometric illustration of a PD-CLF $V(x,\theta)$ for a one-dimensional system with a single parameter ($n=n_\theta=1$ ) in the $x$-$\theta$ plane and an exemplary trajectory of $(x(t),\theta(t),V(t))$ (denoted by the red line) under stabilizing inputs generated by $V$. Notice that the trajectory will eventually approach the $\theta$-axis (corresponding to $x=0$) and stay there regardless of the trajectory of $\theta$. }
    \label{fig:pd-clf-illustration}
    \vspace{-5mm}
\end{figure}

\subsection{Constructing min-norm control law via robust quadratic programs}
Given a valid PD-CLF, at each time $t$, given  $\theta(t)$, $\theta(t)$, $\thetadot(t)$, we can consider all control inputs that satisfy \cref{eq:clf-cond-derivative}:
However, due to the dependence on $\dot\theta$, direct use of \cref{eq:clf-cond-derivative} will lead to a control law dependent on $\dot \theta$ that is not accessible online.  To address this issue, noticing that \cref{eq:clf-cond-derivative} is affine w.r.t. $\dot \theta$, we derive an equivalent condition for \cref{eq:clf-cond-derivative}:
\begin{equation}\label{eq:clf-cond-derivative-sufficient}
\inf_{u}{ L_f V \!+ \!L_g Vu \! +\! \frac{\partial V}{\partial \theta }v } 
\leq -\alpha(\norm{x}),\ \forall (\theta,v) \in \Theta\times\ver{\mcV}, 
\end{equation}
where $\mcV$ is the hyper-rectangular set of admissible $\dot\theta$ values and $\ver{\mcV}$ denote the set of vertices of $\mcV$. As a result, we can construct the control law as
\begin{equation}\label{eq:K-pdclf}
   K_\textup{pd-clf}(x,\theta)  = \! \left\{u \!\in \!\mbR^m\! :  L_f V \!+\! L_g V u\! + \! \frac{\partial V}{\partial \theta }v \!\leq\! -\alpha(\norm{x}), \ \forall v  \!\in  \!\ver{\mcV}\right\}  \!\!
\end{equation}
In particular, we can determine the min-norm control law \cite{freeman1996inverse-clf} by solving a robust QP problem with the constraint \cref{eq:clf-cond-derivative-sufficient} at each time instant $t$: 
\begin{subequations}\label{eq:min-norm-formulation}
  \begin{align}
 &u^\ast(x,\theta) = \min_{u}\  \norm{u}^2 \\
    \textup{s.t. } & L_f V\!+ \!L_g Vu \!+\! \frac{\partial V}{\partial \theta }v 
\leq -\alpha(\norm{x}), \ \forall v \in \textup{ver}(\mcV).   \label{eq:min-norm-formulation-b}
     \end{align}
\end{subequations}
The following lemma establishes the local Lipschitz continuity of the control law obtained by solving \cref{eq:min-norm-formulation}.

\begin{lemma}\label{lemma:qp-control-local-lipschitz}
If $V(x,\theta)$ is a PD-CLF according to \cref{def:pd-clf} and both $\frac{\partial V}{\partial x}$ and $\alpha$ are locally Lipschitz, then the control law $u^*(x,\theta)$ from solving \cref{eq:min-norm-formulation} is locally Lipschitz.     
\end{lemma}
\begin{proof}
    Condition \cref{eq:min-norm-formulation-b} can be rewritten as 
    \begin{equation}\label{eq:clf-deriv-cond-linear-form}
        A(\zeta)^T u \leq c(\zeta) = \min_{i\in |\ver{\mcV}|} b_i(\zeta),
    \end{equation}
  where $\zeta = [x^T, \theta^T]^T$, $A(\zeta) = \left(L_g V \right)^T$, $b_i(\zeta) = - L_f V  -\frac{\partial V}{\partial \theta }v^i-\alpha(\norm{x})$ with $v^i \in \ver{\mcV}$, and  $|\ver{\mcV}|= 2^{n_\theta}$ is the number of vertices of $\mcV$.  
As a result, \cref{eq:min-norm-formulation} can be written as 
\begin{equation}\label{eq:min-norm-formulation-linear-form}
 u^*(\zeta)  = \min_{u}  \norm{u}^2 \quad 
   \textup{s.t. }   A(\zeta)^T u \leq c(\zeta)   
 \end{equation}
The solution for \cref{eq:min-norm-formulation-linear-form} is  given by \cite{freeman1996inverse-clf,ames2016cbf-tac}
\begin{equation}\label{eq:u-zeta-solution}
u^*(\zeta)
\;=\;
\begin{cases}
0, 
   & \text{if } c(\zeta)\ge 0,\\[6pt]
\dfrac{c(\zeta)}{\|A(\zeta)\|^2}\,A(\zeta),
   & \text{if } c(\zeta)< 0.
\end{cases}    
\end{equation}
Note that \cref{eq:clf-cond-derivative-sufficient} indicates that 
\begin{equation}\label{eq:A=0-indicates-c>=0}
    A(\zeta) = 0 \Rightarrow c(\zeta)\geq 0
\end{equation}
As a result, the expression for the case of $c(\zeta)< 0$ in  \cref{eq:u-zeta-solution} is always well defined. 

We next verify that $u^*(\zeta)$ is locally Lipschitz in $\zeta$:

\begin{itemize}
\item \emph{Case $c(\zeta)\ge0$.}
Here $u^*(\zeta)=0$ is a constant function, which is trivially (globally) Lipschitz.

\item \emph{Case $c(\zeta)<0$.} 
Since $A(\zeta)\neq 0$ in this region, $\|A(\zeta)\|\ge \beta>0$ for some local bound $\beta>0$ (by continuity). 
Since the product of locally Lipschiz functions is locally Lipschitz, the maps $\zeta \mapsto A(\zeta)$ and $\zeta \mapsto b_i(\zeta)$ for ${i\in |\ver{\mcV}|}$ are locally Lipschitz. Consequently, $\zeta \mapsto 1/\|A(\zeta)\|^2$ is also locally Lipschitz.  Further considering the fact that the minimum (or maximum) of finitely many Lipschitz functions is still Lipschitz, the map $\zeta \mapsto c(\zeta)$ is also locally Lipschitz. Hence $u^*(\zeta)$ is locally Lipschitz in any neighborhood where $c(\zeta)<0$.

\item \emph{Boundary $c(\zeta)=0$.}
As $c(\zeta)\to 0^-$, $u^*(\zeta) = \frac{c(\zeta)}{\|A(\zeta)\|^2}A(\zeta)\to 0$, matching the solution on the $c(\zeta)\ge0$ side.  Thus there is no jump discontinuity.  Being piecewise linear in $\zeta$ (through the Lipschitz data) and continuous across the boundary, $u^*(\zeta)$ remains locally Lipschitz.
\end{itemize}
\end{proof}
\begin{remark}
In \cite{ames2016cbf-tac}, the authors prove the local Lipschitz continuity of their QP-based controller, which unifies control Lyapunov functions (CLFs) and control barrier functions (CBFs) for  nonlinear time-invariant systems, under the stronger assumption that 
$\frac{\partial V}{\partial x}\neq 0$  everywhere in the region of interest. By contrast, \cref{lemma:qp-control-local-lipschitz} establishes the local Lipschitz continuity without that condition and even for NPV systems.
\end{remark}

\begin{lemma}
   Given an NPV system \cref{eq:dynamics} satisfying \cref{assump:theta-thetadot-range}, if there exists a PD-CLF $V(x,\theta)$ with locally Lipschitz $\frac{\partial V(x,\theta)}{\partial x}$ and $\alpha(\norm{x})$,  then system \cref{eq:dynamics} can be uniformly asymptotically stabilized.    
\end{lemma}
\begin{proof}
  Consider the control law from solving problem \cref{eq:min-norm-formulation}.  Due to the condition \cref{eq:clf-cond-derivative} satisfied by $V(x,\theta)$, and the equivalency between \cref{eq:clf-cond-derivative} and \cref{eq:clf-cond-derivative-sufficient}, problem \cref{eq:min-norm-formulation} is always feasible.  According to \cref{lemma:qp-control-local-lipschitz}, the min-norm control law in \cref{eq:min-norm-formulation} is locally Lipschitz. Considering the closed-loop dynamics under this control law, and \cref{assump:theta-thetadot-range}, we have 
\begin{equation}\label{eq:lyapunov-nonauto-1}
    \dot V =  L_f V + L_g V u +  \frac{\partial V}{\partial \theta }\thetadot\leq -\alpha( \norm{x}),\quad  \forall t\geq0, \forall x.
\end{equation}
Additionally, \cref{eq:clf-cond-V-bnd} indicates that  
\begin{equation}\label{eq:lyapunov-nonauto-2}
    \alpha_1(\norm{x}) \leq V(x(t),\theta(t))  \leq \alpha_2(\norm{x}),\quad  \forall t\geq0, \forall x.
\end{equation}
By applying Lyapunov stability theory for non-autonomous theorem (\cite[Theorem~4.8]{khalil2002nonlinear-book}) to \cref{eq:lyapunov-nonauto-1,eq:lyapunov-nonauto-2}, we conclude that the closed-loop system under the control law in \cref{eq:min-norm-formulation} is uniformly asymptotically stable.   
\end{proof}

\subsection{Robustness of min-norm control law against input gain uncertainty}\label{sec:robustness}
An advantage of the min-norm control law in \cref{eq:min-norm-formulation} is that it is inherently robust against uncertainty in the input gain.  Suppose the original system \cref{eq:dynamics} is perturbed in a way such that the control $u$ is multiplied by a positive gain $\lambda$, i.e., the perturbed system takes the form $\dot x = f(x) + g(x) \lambda u$. Then, the min-norm control law constructed using the nominal dynamics will uniformly asymptotically stabilize the perturbed system for all $\lambda>1$; that is, the closed-loop system with the min-norm control law has an infinite gain margin. This observation is formalized in the following lemma. 
\begin{lemma}\label{lemma:min-norm-robustness}
Suppose $V(x,\theta)$ is a PD-CLF for the system \cref{eq:dynamics}. Then, under the min-norm control law \cref{eq:min-norm-formulation} computed using the nominal dynamics of \cref{eq:dynamics}, the origin of the perturbed system $\dot x = f(x,\theta) + g(x,\theta) \lambda u$ for all $\lambda>0$ is uniformly asymptotically stable for all admissible trajectories of $\theta$. 
\end{lemma}
\begin{proof}
     Note that the CLF derivative condition \cref{eq:min-norm-formulation-b} can be rewritten as \cref{eq:min-norm-formulation-linear-form}, repeated below:
     \begin{equation}\label{eq:clf-deriv-cond-linear-form-2}
        A(\zeta)^T u \leq c(\zeta) ,
    \end{equation}
  Now consider the condition
    \begin{equation}\label{eq:clf-deriv-cond-linear-form-perturb}
        A(\zeta)^T \lambda u \leq c(\zeta) ,
    \end{equation}
  If $c(\zeta)\geq 0$, from \cref{eq:u-zeta-solution}, we know that the solution for \cref{eq:min-norm-formulation} is $u^* = 0$, which obviously also satisfy \cref{eq:clf-deriv-cond-linear-form-perturb}. \\  
  If $c(\zeta)< 0$, then $A(\zeta)^T u \leq c(\zeta)<0$, which indicates that $A(\zeta)^T \lambda u \leq A(\zeta)^T u \leq c(\zeta)<0$ for all $\lambda\geq 1$.  
  Notice that \cref{eq:clf-deriv-cond-linear-form-perturb} is an equivalent reformulation of 
  \begin{equation}\label{eq:clf-deriv-cond-perturb}
      L_f V \!+\! L_g V (\lambda u)\! + \! \frac{\partial V}{\partial \theta }v \!\leq\! -\alpha(\norm{x})
  \end{equation}
Additionally, \cref{eq:clf-cond-V-bnd} indicates that \cref{eq:lyapunov-nonauto-2}. By applying Lyapunov stability theory for non-autonomous theorem (\cite[Theorem~4.8]{khalil2002nonlinear-book}) to \cref{eq:lyapunov-nonauto-1,eq:lyapunov-nonauto-2}, we conclude that the perturbed system under the control law in \cref{eq:min-norm-formulation} is uniformly asymptotically stable.   
\end{proof}

\subsection{Local analysis and PD region of stabilization} \label{sec:pd-clf-local-analysis}
We now consider local stabilizability ensured by a PD-CLF. 

\begin{definition}\label{def:pd-clf-local}
  A continuously differentiable function $V(x,\theta)$ is a local PD-CLF for \cref{eq:dynamics}  if there exist continuous positive definite functions $\alpha_1$, $\alpha_2$, and  $\alpha$  such that the following conditions hold:  

\begin{subequations}\label{eq:clf-cond-local}
    \vspace{-4mm}
    \begin{align}
     &   \hspace{-5mm}  \alpha_1(\norm{x}) \leq V(x,\theta)  \leq \alpha_2(\norm{x}),~ \forall (x,\theta)\in \mcX \times \Theta, 
     \label{eq:clf-cond-V-bnd-local}   \\
   &  \hspace{-5mm} 
   \forall (x,\theta,\thetadot)\in \mcX \times \Theta\times \mcV, \quad  \nonumber \\     
  &  \hspace{3mm} \inf_u \underbrace{ \frac{\partial V}{\partial x}f + \frac{\partial V}{\partial x}gu + \frac{\partial V}{\partial \theta }\thetadot}_{\trieq \dot V(x,\theta,\thetadot,u)} \leq -\alpha(\norm{x}),\label{eq:clf-cond-derivative-local} 
       \vspace{-8mm}
    \end{align}
\end{subequations} 
where $\mcX\subset\mbR^n$ is a compact set containing the origin in its interior. 
\end{definition}

\begin{definition}\label{def:pd-ros}
A PD region of stabilization (PD-ROS) for system \cref{eq:dynamics} is the set of initial states $x(0)$ and parameters $\theta(0)$ from which there exists a control law $K(x,\theta)$ such that $u(t)=K(x(t),\ \forall t\geq 0$, and 
$\lim_{t\rightarrow 0} x(t) = 0$ 
for all admissible trajectories of $\theta$. Formally, a PD-ROS is a set defined by \\
$
\left\{(x(0),\theta(0)) \in \mbR^{n}\times \Theta\!:  \ \forall \theta \in \fThetaV,  \exists u(t)= K(x(t),\theta(t))\  \forall t\geq 0,   
\textup{ s.t.}  \lim_{t\rightarrow 0} x(t) = 0.
\right\}
$

\end{definition}

\begin{definition}\label{def:r-ros}
A robust region of stabilization (R-ROS) for system \cref{eq:dynamics} is the set of initial states $x(0)$  from which there exists a control law $K(x,\theta)$ such that $u(t) = K(x(t),\theta(t))$ $\forall t\geq 0$, and $\lim_{t\rightarrow 0} x(t) = 0$ for all admissible trajectories of $\theta$. Formally, an R-ROS is a set defined by \\
$\left\{x(0) \in \!\mbR^{n}\!\!:  \ \forall \theta \in \fThetaV,\  \exists u(t)= K(x(t),\theta(t))\  \forall t\geq 0, 
\textup{ s.t.}  \lim_{t\rightarrow 0} x(t) = 0.
\right\}$

\end{definition}
\begin{remark}
    Compared to R-ROS, PD-ROS allows us to enlarge the set of states from which the system can be stabilized by explicitly considering $\theta(0)$, the initial value of $\theta$. This is demonstrated by \cref{fig:roa-pendulum} to be presented in \cref{sec:sim}. 
\end{remark}

For future use, we  define
   \begin{equation}\label{eq:Omega-rho-defn}
  \Omega(V,\rho)\! = \!\left \{(x,\theta)\in\mbRn \!\times\! \mbR^\ntheta: V(x,\theta)\leq \rho, \theta\in \Theta \right\},
\end{equation} 
where $\rho$ is a positive constant.

\begin{lemma}\label{lem:omega-is-pd-ros}
   Supposing $V(x,\theta)$ is a local PD-CLF for \cref{eq:dynamics} as defined in \cref{def:pd-clf-local},  if $\rho$ is selected such that $\Omega(V,\rho)\subseteq \mcX\times\Theta$, then $\Omega(V,\rho)$ is a PD-ROS. 
\end{lemma}
\begin{proof}
    Condition \cref{eq:clf-cond-derivative-local} ensures that  if $(x(0),\theta(0))\in \Omega(V,\rho)$, then there exist $u(t)$ such that 
    \begin{equation}
         \dot V(t)\leq -\alpha( \norm{x(t)}), \ \forall t\geq 0,
    \end{equation}
   $\forall \theta \in\! \fThetaV$. As a result, $V(x(t),\theta(t))\leq V(x(0),\theta(0)\leq \rho$, $\forall t\geq 0$, which implies $x(t)\in \mcX$, $\forall t\geq 0$. By applying \cite[Theorem~4.8]{khalil2002nonlinear-book}, one can immediately conclude that the origin can be uniformly asympotically stabilized, and thus $\lim_{t\rightarrow 0} x(t) \!=\! 0$. Therefore, $\Omega(V,\rho)$ is a PD-ROS.
\end{proof}

\section{Convex synthesis of PD-CLFs Using SOS Programming}
In this section, we consider the system in \cref{eq:dynamics} to have a polynomial form, i.e., $f(x,\theta)$ and $g(x,\theta)$ are polynomial functions. Furthermore, we assume that the nonlinear polynomial system can be rewritten in  
    a linear-like form \cite{prajna2004nonlinear-sos}:
\begin{equation}\label{eq:dynamics-linear-form}
    \dot x = A(x,\theta) x + B(x,\theta) u,
\end{equation}
where $A(x,\theta)\in \mbR^{n\times n}$ is a polynomial matrix, $B(x,\theta) = g(x,\theta)$. 
\begin{remark}
Note that a polynomial system \cref{eq:dynamics} can always be written in the form of \cref{eq:dynamics} if $f(x,\theta)$ does not contain terms independent of $x$, i.e., monomials of zero degrees in $x$. When $f(x,\theta)$ contains $x$-independent terms, state transformation may be leveraged to obtain a linear-like form, {as we will demonstrate in \cref{sec:rocket}}. 
\end{remark}  

We consider local synthesis, i.e., assuming $x\in\mcX$, where $\mcX$ is a set defined by 
\begin{align}
{\mcX} &\trieq  \{x\in\mbRn: |{C}_i(x)x)|\leq 1,\ i\in\mbZ_1^{n}\},  \label{eq:X-ubar-defn}\\
& = \{x\in\mbRn\!: c_i(x)\geq 0,\ i\in\mbZ_1^{n}\}, \label{eq:X-set-defn-npv}
\vspace{2mm}
\end{align}
where $C_i:\mbR^n\rightarrow\mbR^{1\times n}$ is a vector-valued polynomial function, and $c_i(x) = 1-x^TC_i^TC_i x$ is a polynomial function.
\subsection{Preliminaries about SOS optimization}
A polynomial $l(x)$ is an SOS if there exist polynomials $l_1(x),\dots,l_m(x)$
 such that $l(x) = \sum_{i=1}^{m}l_i^2(x)$. A polynomial $l(x)$ of degree $2d$ is an SOS iff there exists a {positive semidefinite (PSD)} matrix $Q$ such that $l(x)= y^T(x)Q  y(x)$, where $y(x)$ is a column vector whose entries are all monomials in $x$ with degree up to $d$ \cite{parrilo2000structured-sos}. 
 An SOS decomposition for a given $l(x)$ can be computed using semidefinite programming (SDP) (by searching for a PSD matrix $Q$) \cite{prajna2004nonlinear-sos}. 
 Hereafter, we use $\sos{x}$ to denote the set of SOS polynomials in $x$ and $\sos{x,y}$ to denote the set of SOS polynomials in $x$ and $y$.  


 \begin{proposition} \label{proposition:s-procedure} (S-procedure)  \cite{parrilo2000structured-sos}
 Let $p(x)$ be a a polynomial and  let $\mcX$ be a set defined in \cref{eq:X-set-defn-npv}. Suppose there exist SOS polynomials $\lambda_i(x)$, $i=0,1,\dots,n$ such that 
 $\left(1+\lambda_0(x)\right)p(x)- \sum_{i=1}^{n} \lambda_i(x) c_i(x) \in \sos{x}$. Then, $p(x)\geq 0 \textup{ for all } x\in\mcX$. 
 \end{proposition}
\begin{remark}
    In practice, $\lambda_0(x)$ is often selected to be 0.
\end{remark}

 \begin{proposition}\cite[Proposition~2]{prajna2004nonlinear-sos}\label{proposition:sos-psd}
 Let $F(x)$ be an $N\!\times\! N$ symmetric polynomial matrix of degree $2d$ in $x\!\in\!\mbR^n$. Then, $F(x)\!\geq \!0 
    \textup{ } \forall x\in\mbR^n$, if $v^T\!F(x)v\!\in\! \sos{x,v} , \textup{ where } v\!\in\! \mbR^N$. 
 \end{proposition}

 \begin{proposition}\cite[Proposition~10]{prajna2004nonlinear-sos}\label{proposition:sos-psd-local}
 Let $F(x)$ be an $N\times n$ symmetric polynomial matrix of degree $2d$ in $x\in\mbR^n$, $\mcX$ is a semialgebraic set defined in \cref{eq:X-set-defn-npv}, and  
 $v\in \mbR^N$. Suppose there exist SOS polynomials $\lambda_i(x,v)$, $i=1,\dots,n$ 
 such that 
 $v^TF(x)v - \sum_{i=1}^n \lambda_i(x,v) c_i(x) \in \sos{x,v}$. Then, $F(x)\geq 0 \textup{ for all } x\in\mcX$. 
 \end{proposition}

\subsection{Convex Synthesis of a PD-CLF and a PD Controller}
We present an SOS optimization-based method for the joint synthesis of a PD-CLF and a PD controller. The method is inspired by \cite{prajna2004nonlinear-sos,maier2010predictive-sos,zhao2023convex-cbf}.

Let $A_j(x,\theta)$ denote the $j$-th row of $A(x,\theta)$, $J=\left\{j_1,j_2,\dots, j_m\right\}$ denote the row indices of $g(x,\theta)$ whose corresponding row is equal to zero, and define $$\tilde x =(x_{j_1}, x_{j_2},\dots, x_{j_m}),$$ {which includes all the states whose derivatives are not directly affected by control inputs.}

The result is summarized in the following theorem.
\begin{theorem}\label{them:npv-synthesis-sos}
Consider an NPV system \cref{eq:dynamics-linear-form} with the parameters $\theta$ satisfying \cref{eq:theta-thetadot-constrs}, 
and the set $\mcX$ defined in \cref{eq:X-set-defn-npv}. Suppose there exists polynomial matrices $X(\tilde x, \theta)\in \mcS^n$ and $Y(x,\theta)\in \mbR^{m\times n}$, constants $\varepsilon_i>0$ ($i=1,2,3$), and SOS polynomials 
    $\lambda_i^l$ ($i\in\mbz{\Ntheta}, l\in\mbz{4}$), $\alpha_i^j ~(i\in\mbz{\Ntheta},j\in\mbz{n})$,  
    $\mu_j^p$ ($j\in\mbz{n}, p\in\mbz{3}$)
    and $\xi_k$ ($k\in\mbz{\ntheta}$)     
    such that 
    \begin{subequations}\label{eq:joint-syn-sos}
    \begin{align}
& v_1^T\left(X(\tilde x,\theta)-\varepsilon_1 I\right) v_1 - \sum_{i=1}^\Ntheta \lambda_i^1 h_i(\theta) -\! \sum_{j=1}^{n} \mu_j^1 c_j(x)  \in \sos{x,\theta,v_1}
\label{eq:joint-syn-sos-1} \\
& v_1^T\left(\varepsilon_2 I - X(\tilde x,\theta)\right) v_1 - \sum_{i=1}^\Ntheta \lambda_i^2 h_i(\theta)  -\! \sum_{j=1}^{n} \mu_j^2 c_j(x)  \in \sos{ x,\theta,v_1} \label{eq:joint-syn-sos-2}\\
  &  - v_1^T\!\left( F_1(x,\theta,\thetadot)
   +\varepsilon_3I \right) v_1 - 
   \sum_{i=1}^\Ntheta \lambda_i^3 h_i(\theta)  -\! \sum_{j=1}^{n} \mu_j^3 c_j(x) - \!\sum_{k=1}^\ntheta \xi_k (\dot \theta_k\! -\! \ubar v_k)  (\bar v_k\!-\!\dot \theta_k) \in \sos{x,\theta,\thetadot,v_1}, \label{eq:joint-syn-sos-3} \\
& 
{\setlength{\arraycolsep}{2pt}
v_2^T \!\begin{bmatrix}
1 & {C_j} X(\tilx,\theta)\\
* & X(\tilx,\theta)    \end{bmatrix}}v_2\!- \!\sum_{i=1}^\Ntheta \alpha_i^j h_i(\theta) 
\!\in\! \sos{x,\theta,v_2},\ \forall j\in \mbZ_1^{n}\!, \label{eq:joint-syn-sos-state-cst}
\end{align}  
\end{subequations}
where 
\begin{equation}\label{eq:F1-defn}
    F_1(x,\theta,\thetadot) = \sym{AX\!+\!BY}-\! \sum\limits_{j \in J} \frac{\partial X}{\partial x_j} \left( {{A_j}x}\right) \!- \!\sum_{k=1}^\ntheta \frac{\partial X}{\partial \theta_k} \dot \theta_k, 
\end{equation} and $v_1,~v_2,$ and $v_3$ are vectors of proper dimensions. 
Then,
\begin{enumerate}[label=(\alph*) ]
    \item the controller 
\begin{equation}\label{eq:control-law-Y-X}
     u(x,\theta) = Y(x,\theta)X^{-1}(\tilde x,\theta)x
\end{equation}
renders the closed-loop system exponentially stable (GES) 
for all admissible trajectories of $\theta$ satisfying \cref{eq:theta-thetadot-constrs}. Moreover, 
\begin{equation}\label{eq:V-defn-w-X}
    V(x,\theta) = x^T X^{-1}(\tilde x,\theta)x
\end{equation}
is a PD Lyapunov function to certify such stability. 
\item The set $\Omega(V,1) = \{(x,\theta)\in\mbR^{n} \times \mbR^{\ntheta}: V(x,\theta)\leq 1, \theta\in \Theta\}$ is a subset of a PD region of attraction, i.e., $\forall(x(0),\theta(0))\in \Omega(V,1)$,  $\lim_{t\rightarrow 0} x(t) = 0$ for all admissible trajectories of $\theta$ satisfying \cref{eq:theta-thetadot-constrs}.  
\end{enumerate}
\end{theorem}

\begin{proof} 
{\it Part (a)}: 
 According to \cref{proposition:sos-psd-local}, \cref{eq:joint-syn-sos-1,eq:joint-syn-sos-2} ensures  
 \begin{equation}\label{eq:X-bnd-epsilon12}
     \varepsilon_1 I \leq  X(\tilde x,\theta)\leq \varepsilon_2I,\ \forall \theta\in\Theta,\  x\in \mcX,
 \end{equation}
 which implies 
 \begin{equation}\label{eq:joint-syn-V-bnd}
     \frac{1}{\varepsilon_2} \norm{x}^2  \leq V(x,\theta) \leq \frac{1}{\varepsilon_1} \norm{x}^2,\ \forall \theta\in\Theta, x\in \mcX.
 \end{equation}

Under the dynamics \cref{eq:dynamics-linear-form} and the control law \cref{eq:control-law-Y-X}, and the definition of $V$ in \cref{eq:V-defn-w-X}, we have 
{
{
\begin{align}
  & \dot V(x,\theta)  = \sym{\!x^T X^{-1} \dot x\!}  + x^T\!\left [ \sum\limits_{j \in J} \!\frac{\partial{ X^{ - 1}}}{\partial x_j} \dot x_j \!+\! \sum_{k=1}^\ntheta \! \frac{\partial X^{-1}}{\partial \theta_k} \dot \theta_k\! \right ]\!x\nonumber    
  \\
  & = \!{x^T}\!X^{-1} \!\!\left[\sym{ {A X \!+\!{B}{Y}}} -  \!\! 
  \sum\limits_{j \in J}\frac{\partial{ X}}{\partial {x_j}} ( {A_jx} ) -   \!\!\sum_{k=1}^\ntheta \! \frac{\partial X}{\partial \theta_k} \dot \theta_k \! \right]\! X^{-1}x \label{eq:Vdot-expression-1} \\ 
    &= \! {x^T}{X^{ - 1}}F_1(x,\theta,\thetadot){X^{ - 1}}x \label{eq:Vdot-expression-2}
\end{align}}}%
where we have leveraged the fact that $\dot x =\left( {A + BY{X^{ - 1}}} \right)x$ and $\frac{{\partial {X^{ - 1}}}}{{\partial {x_j}}} =  - {X^{ - 1}}\frac{{\partial X}}{{\partial {x_j}}}{X^{ - 1}}$ \cite[Lemma 5]{prajna2004nonlinear-sos} in deriving \cref{eq:Vdot-expression-1}, and $F_1(x,\theta,\thetadot)$ is defined in \cref{eq:F1-defn}. According to \cref{proposition:sos-psd-local}, \cref{eq:joint-syn-sos-3} implies  $F_1(x,\theta,\thetadot)+\varepsilon_3I\leq 0$, $\forall x\in\mcX$, $\theta\in \mcF_\Theta^\mcV$ with $\mcF_\Theta^\mcV$ defined in \cref{eq:F-Theta-mcV-defn}. Further considering \cref{eq:Vdot-expression-2}, we have 
\begin{align}
    \dot V(x,\theta)   &\leq -\varepsilon_3 x^T X^{-1} X^{-1} x \leq -\frac{\varepsilon_3}{\varepsilon_2^2} \norm{x}^2,\  \forall x\in\mcX,\  \theta\in \fThetaV, \label{eq:Vdot-expression-final}
\end{align}
where the last inequality is due to $X^{-1}\ge \frac{1}{\varepsilon_2}I$, $\forall x\in\mcX$, $\theta\in\Theta$ (\cref{eq:X-bnd-epsilon12}).   

By applying Lyapunov stability theory \cite[Theorem~4.10]{khalil2002nonlinear-book} and considering \cref{eq:joint-syn-V-bnd,eq:Vdot-expression-final}, we can conclude that under the controller \cref{eq:control-law-Y-X}, the closed-loop system is locally exponentially stable at the origin, as certified by the Lyapunov function defined in \cref{eq:V-defn-w-X}. 

{\it Part (b)}: According to \cref{proposition:sos-psd-local}, \cref{eq:joint-syn-sos-state-cst} implies ${\footnotesize\setlength{\arraycolsep}{1.9pt} \begin{bmatrix}
1 & {C_i} X\\
* & X
    \end{bmatrix}} \geq 0$,  $\forall \theta \in \Theta$,  which, via Schur complement, further indicates $XC_i^T{C_i}X\leq X$, $\forall \theta \in \Theta$. Multiplying the preceding inequality by $x^TX^{-1}$ and its transpose from the left and right, respectively, leads to 
    $
        |{C_i}x|^2\leq x^TX^{-1}x = V(x,\theta),~\forall i\in \mbZ_1^{n},\ \forall \theta\in \Theta$. 
    Thus, $\forall (x,\theta)\in \Omega(V,1)$, we have  
    $|{C_i}x|^2\leq x^TX^{-1}x\leq 1$, $~\forall i\in \mbZ_1^{n}$, which indicates $x\in\mcX$. Further considering \cref{eq:Vdot-expression-final}, we can conclude that if $ (x(0),\theta(0))\in \Omega(V,1)$, $\theta\in\fThetaV$, then
    \begin{equation}\label{eq:OmegaV-indicates-X} 
     (x(t),\theta(t))\in\Omega(V,1),\  x(t)\in\mcX,\  \forall t\geq 0.
    \end{equation}
     In other words, under the condition that $\theta\in\fThetaV$, $\Omega(V,1)$ is an invariant set, and the projection of $\Omega(V,1)$ onto the $x$-plane is a subset of $\mcX$. 
     As a result, the conditions \cref{eq:Vdot-expression-final,eq:joint-syn-V-bnd} indicate $ \dot V \leq  -\frac{\varepsilon_1\varepsilon_3}{\varepsilon_2^2} V$, which implies $V(t)\leq V(0)\exp^{-\frac{\varepsilon_1\varepsilon_3}{\varepsilon_2^2} t}$, $\forall t\geq 0$. Considering \cref{eq:joint-syn-V-bnd}, we have $\norm{x(t)}^2\leq \varepsilon_2V(t)\leq \varepsilon_2 V(0)\exp^{-\frac{\varepsilon_1\varepsilon_3}{\varepsilon_2^2} t}$, $\forall t\geq 0$, which indicates $\lim_{t\rightarrow 0} x(t) = 0$. 
\end{proof}
\begin{remark}
    \cref{them:npv-synthesis-sos} presents the conditions to synthesize a controller to ensure the exponential stability of the closed-loop system, which is stronger than the uniformly asymptotic stability adopted in \cref{sec:pd-clf,sec:pd-clf-local-analysis}. Additionally, condition \cref{eq:clf-cond} are convex SOS conditions since they are affine in all decision variables. 
\end{remark}
\begin{remark}
    Similar to \cite{prajna2004nonlinear-sos}, the matrix $X$ in \cref{them:npv-synthesis-sos} depends only on $\tilx$, whose time derivatives are not affected by control inputs. This is for ensuring $\dot V$ does not depend on control input $u$ in order to obtain convex SOS conditions. 
\end{remark}

The following lemma states that the Lyapunov function from \cref{them:npv-synthesis-sos} is a  PD-CLF and the associated PD region of attraction is a PD-ROS. The proof is straightforward and omitted. 
\begin{lemma}\label{lemma:LF-is-CLF}
    Suppose the system \cref{eq:dynamics-linear-form} is an equivalent representation of the NPV system \cref{eq:dynamics},
    and all the conditions of \cref{them:npv-synthesis-sos} hold. Then the PD Lyapunov function defined in \cref{eq:V-defn-w-X} is a PD-CLF for \cref{eq:dynamics}, and $\Omega(V,1)$ is a PD-ROS.
\end{lemma}

To maximize the volume of $\Omega(V,1)$, we maximize the volume of an elliptic cylinder $\Omega(x^TX_0^{-1}x,1)=\left \{(x,\theta)\in\mbR^{n}\times \mbR^{\ntheta}|x^TX_0^{-1}x\leq 1,\theta\in\Theta \right\}$ contained inside $\Omega(V,1)$, where $X_0$ is a positive definite matrix. Formally, we solve the following SOS problem:
{
\begin{subequations}\label{eq:opt-max-ellipsoid-joint-syn}
    \begin{align}
    &\max_{
    \footnotesize
    \begin{array}{c}
       X_0>0, \lambda_i^4 (i\in\mbz{N_\theta}), 
       \mu_j^4 (j\in\mbz{n})
    \end{array}
    } \log \det(X_0) \\
     \textup{s.t. } &  v_1^T(X \!- \!X_0)v_1 \!-\! \!\sum_{j=1}^{n} \mu_j^4 c_j(x) \!-  \!\!\sum_{i=1}^\Ntheta \lambda_i^4 h_i(\theta) \in\sos{x,\theta,v_1} \label{eq:opt-max-ellipsoid-constraint-joint} \\
   &   \lambda_i^4,  \mu_j^4 \in\sos{x,\theta,v_1},\ \forall i\in\mbz{N_\theta},  j\in\mbz{n}
     \end{align}
\end{subequations}}
Condition \cref{eq:opt-max-ellipsoid-constraint-joint} ensures $X\geq X_0$ and therefore $x^TX^{-1}x \leq x^TX_0^{-1}x$ for any $x\in\mcX,\ \theta\in\Theta$. As a result, $\Omega(x^TX_0^{-1}x,1) \cap (\mcX\times \Theta) \subseteq \Omega(V,1) \cap (\mcX\times \Theta) $.

To summarize, to obtain a PD  controller (a PD-CLF) to exponentially stabilize the system 
while maximizing the PD-ROA (PD-ROS), we solve the following SOS program: \vspace{1 mm}\\ 
{
\fbox{%
\parbox{0.99\columnwidth}{%
\begin{subequations}\label{eq:opt-joint-syn-sos}
  \begin{align}
    & \hspace{-8mm}\max_{
    \footnotesize
    \begin{array}{c}
      \varepsilon_i\!>\!0~(i\in\mbz{3}), X_0\!>\!0,   \alpha_i^j ~(i\in\mbz{\Ntheta},j\in\mbz{n})
      \\
      X(\tilde x,\theta)
      , Y(x,\theta),
      \lambda_i^l\ (i\in\mbz{\Ntheta}\!, l\in\mbz{6}),  
      \\      
      \mu_j^p\  (i\in\mbz{\Ntheta}, p\in\mbz{4}), \xi_k\ (k\in\mbz{\ntheta}) 
    \end{array}
    } \!\!\log \det(X_0) \label{eq:opt-joint-syn-sos-1} \\
     \textup{s.t. } & \textup{Constraints \cref{eq:joint-syn-sos-1,eq:joint-syn-sos-2,eq:joint-syn-sos-3,eq:joint-syn-sos-state-cst,eq:opt-max-ellipsoid-constraint-joint}}   
     \\
     & \textup{all multipliers are SOS}.
\end{align}
\end{subequations}} 
}

\subsection{QP-based control law vs pre-defined control law}
Since the resulting Lyapunov function  from solving \cref{eq:opt-joint-syn-sos} is a PD-CLF according to \cref{lemma:LF-is-CLF}, we can again construct the control signal by solving a QP similar to \cref{eq:min-norm-formulation}. From \cref{eq:Vdot-expression-final}, the QP becomes
\begin{subequations}\label{eq:min-norm-formulation-w-u-limit}
  \begin{align}
 &u^\ast(x,\theta) = \min_{u}\  u^T H(\theta) u \\
    \textup{s.t. } & L_f V\!+ \!L_g Vu \!+\! \frac{\partial V}{\partial \theta }v 
\leq -\varepsilon_3 x^T X^{-1} X^{-1} x. \label{eq:min-norm-formulation-w-u-limit:b}
     \end{align}
\end{subequations}
where we also introduce a $\theta$-dependent positive definite matrix $H(\theta)$ to make the cost function more general.  Notice that \label{eq:min-norm-formulation-w-u-limit:b} can be replaced by $L_f V\!+ \!L_g Vu \!+\! \frac{\partial V}{\partial \theta }v  \leq -\frac{\varepsilon_3}{\varepsilon_2^2} \norm{x}^2, \ \forall v \in \textup{ver}(\mcV)$, which is a more relaxed condition and will lead to more conservative behavior.  
\begin{remark}
 The QP problem  \cref{eq:min-norm-formulation-w-u-limit} is guaranteed to be feasible for any $(x,\theta)\in\Omega(V,1)$ due to \cref{them:npv-synthesis-sos}. 
\end{remark}
Since an input control law defined in \cref{eq:control-law-Y-X} is simultaneously obtained when solving \cref{eq:opt-joint-syn-sos} to obtain the PD-CLF, one may wonder why we still need to resort to the QP \cref{eq:min-norm-formulation-w-u-limit} to construct the min-norm control law. The min-norm may be preferred when avoiding input saturation is a concern, as min-norm control law typically generates small control signal as demonstrated in case studies in \cref{sec:sim}). Additionally, min-norm control law may improve the robustness of the closed-loop system against input uncertainty, as demonstrated in \cite{jankovic1999clf-robustness}. 

 \section{Simulation Results}\label{sec:sim}
We used Matlab 2023a with YALMIP \cite{YALMIP} and Mosek \cite{andersen2000mosek} 
to solve all the optimization problems.

\subsection{Numerical example}
We first consider a two-dimensional toy example 
\begin{equation}
   \dot x = \begin{bmatrix}
      \dot x_1 \\
      \dot x_2 
   \end{bmatrix} = \begin{bmatrix}
      \theta(t) x_2 \\
       19.62\left(x_1-{x_1^3}/{6}\right)-8x_2
   \end{bmatrix} 
   + \theta(t)  \begin{bmatrix}
       0 \\
      40
   \end{bmatrix} u.
\end{equation}
where $\theta(t)\in\Theta = [0.05,1]$ and $|\dot\theta(t)|\leq 0.1$. Note that $\theta=0.05$
 indicates a loss of 95\% control effectiveness compared to the scenario with $\theta=1$ from the input matrix. To find a PD-CLF, we represented the dynamics in the linear-like form \cref{eq:dynamics-linear-form} and then solved the SOS problem \cref{eq:opt-joint-syn-sos} to maximize the region of stabilization. We selected the decision variables as $X(x_1,\theta)$ and $Y(x,\theta)$, both of which have the order of 2, and the set $\mcX$ to be $[-5,5]\times[-5,5]$.
 The resulting PD-CLF is given by $V(x,\theta) = x^T X^{-1}(x_1,\theta)x$, where 
$$  X(x_1,\theta) =  \begin{bmatrix}
 0.81 & 
-1.79 \\
-1.79 & 24.82       
   \end{bmatrix}  +
   \begin{bmatrix}
-1.16 & 
-3.32 \\
-3.32 & -0.34      
   \end{bmatrix} \theta   + 
10 
^{-4}       \begin{bmatrix}
7 & -114\\
 -114 & 172 
\end{bmatrix} x_1^2 
+ \begin{bmatrix}
12.12 & -0.50 \\
-0.50 & 0.18   
\end{bmatrix} \theta^2.  $$
Note that solving the optimization problem \cref{eq:opt-joint-syn-sos} also gives a control law 
\begin{equation}\label{eq:control-law-pendulum}
   u(x,\theta) = Y(x,\theta)X^{-1}(x_1,\theta)x.
\end{equation}
  The PD region of stabilization (ROS) determined by $V(x,\theta)$ as $\Omega(V,1)=\{(x,\theta): V(x,\theta)\leq 1, \theta\in\Theta\}$ is illustrated in \cref{fig:roa-pendulum}. \cref{fig:roa-pendulum} also includes the projection of  $\Omega(V,1)$ onto the $x$-plane with different $\theta$ values as well as the robust ROS, which is obtained by removing the dependence of the decision variables $X$ and $Y$ on $\theta$ in solving the SOS problem \cref{eq:opt-joint-syn-sos}. We can see that the robust ROS almost overlaps with the projection of the PD-ROS onto $\theta =0.05$, which is the smallest among all projected regions.  Using the PD-ROS leads to an enlarged set of states from which the system can be stabilized by considering the initial values of $\theta$, i.e., $\theta(0)$.
\begin{figure}[h]
    \centering
    \vspace{-5mm}
\includegraphics[width=0.45\columnwidth]{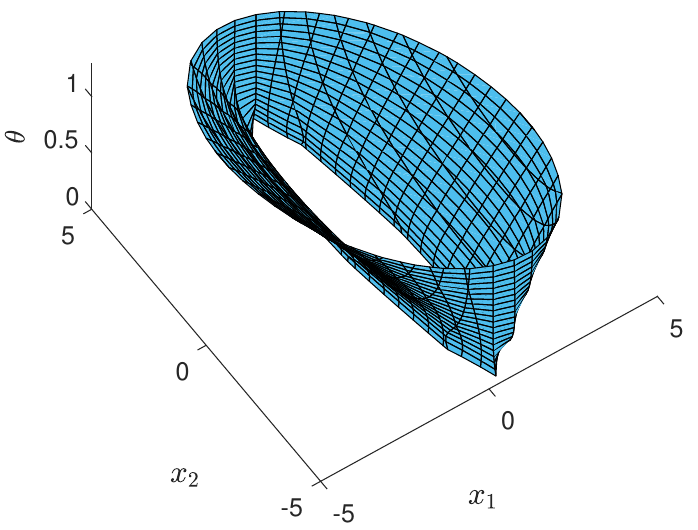} 
\includegraphics[width=0.47\columnwidth]{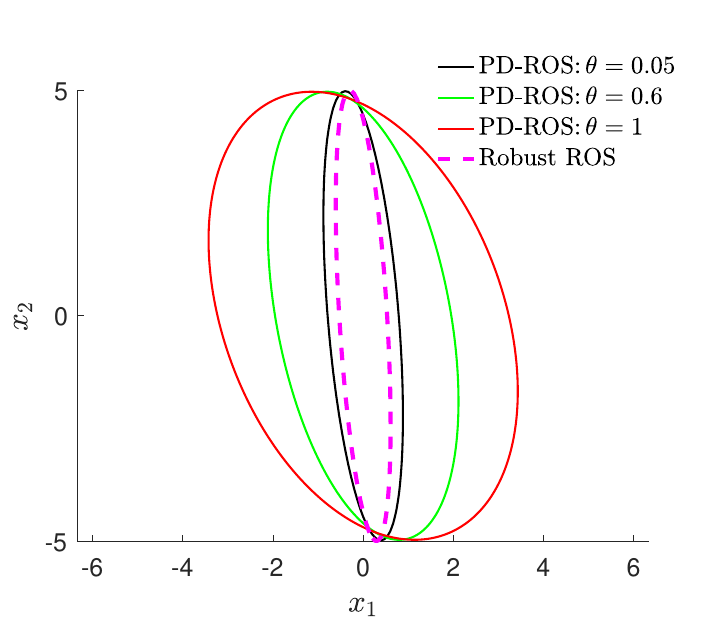} 
    \caption{PD region of stabilization (PD-ROS) (left) and its projection onto $x$-planes with different $\theta$-values together with the robust ROS (right). Robust ROS overlaps with the smallest projection of the PD-ROS, corresponding to $\theta=0.0.05$}
    \label{fig:roa-pendulum}
    \vspace{-5mm}
\end{figure}

To demonstrate this, we set the parameter trajectory to be $\theta(t) = 0.6+0.4\cos(0.2t)$, which leads to $\theta(0) = 1$. We pick four initial states $x(0)$ such that $(x(0)$ and $\theta(0))$ are in the PD-ROS, i.e., $(x(0),\theta(0))\in \Omega(V,1)$, but are not in the robust ROS. \cref{fig:toy-traj} illustrates the trajectories of states, Lyapunov function, and control input under the control law \cref{eq:control-law-Y-X} and the min-norm control law, respectively. Under both control laws, the system is successfully stabilized to the origin for all four initial states. However, the pre-defined control law and min-norm control law . In particular, the min-norm control inputs are small at the very beginning. In fact, the min-norm inputs are zero at the first 0.03 seconds for all four initial states. This is possible because the CLF condition is satisfied by the dynamics even in the absence of active control input at the very beginning. 
\begin{figure}[H]
    \centering
    \begin{subfigure}[b]{0.42\columnwidth}
        \centering
\includegraphics[width= 1\columnwidth]{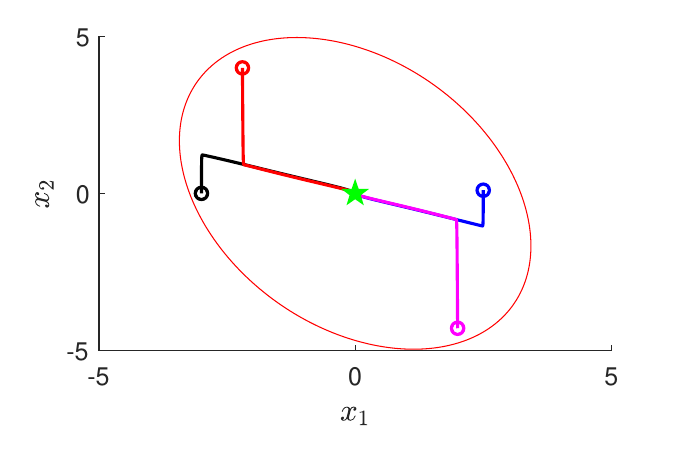} \\ 
\includegraphics[width= 1\columnwidth]{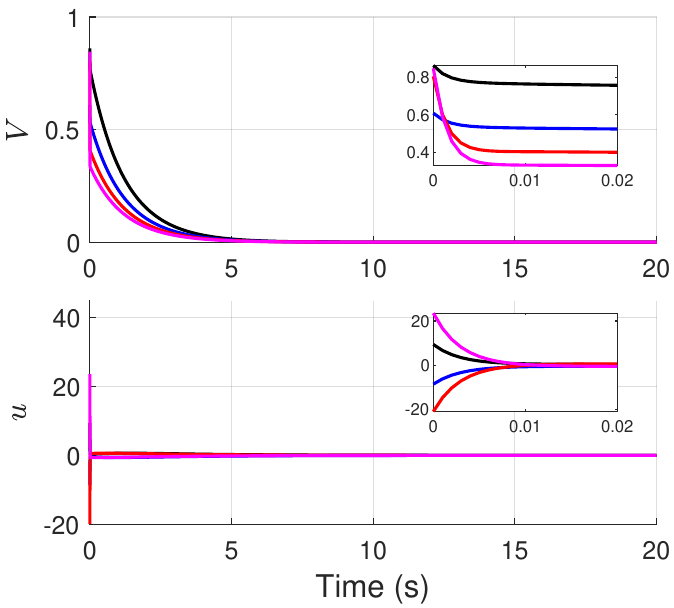}
        \caption{Pre-defined control law \cref{eq:control-law-Y-X}}
        \label{fig:toy-predefine-u}
    \end{subfigure}
      \begin{subfigure}[b]{0.42\columnwidth}
    \includegraphics[width=1\columnwidth]{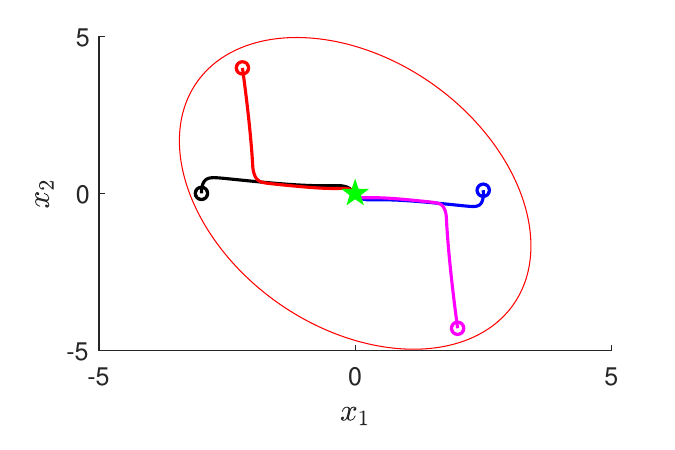} \\ 
\includegraphics[width= 1\columnwidth]{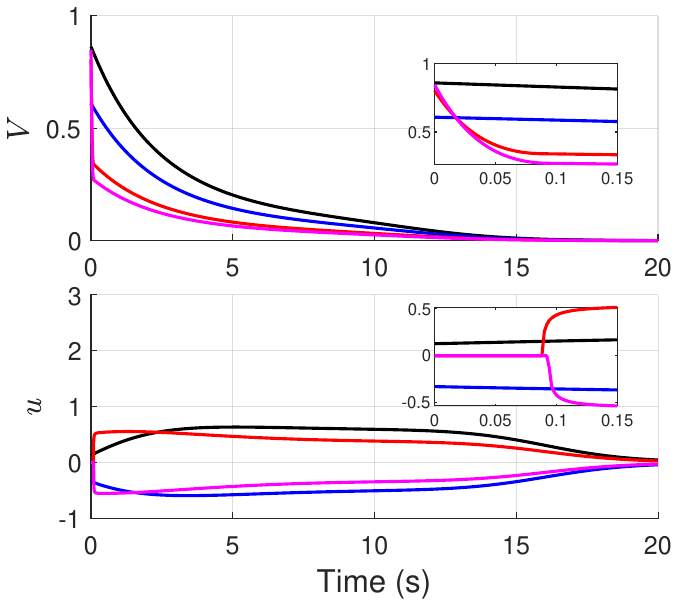}
        \caption{Min-norm control law \cref{eq:min-norm-formulation}}
        \label{fig:toy-predefine-u}
    \end{subfigure}
    \caption{Trajectories of states (top), Lyapunov function (middle) and control input (bottom) starting from four initial points under the pre-defined control law \cref{eq:control-law-Y-X} (left) and the min-norm control law \cref{eq:min-norm-formulation} (right).}
    \label{fig:toy-traj}
\end{figure}

\subsection{Rocket landing under varying mass and inertia   }\label{sec:rocket}
We now consider a 2-D rocket landing problem. The dynamics of the rocket are described by

\begin{equation}\label{eq:rocket-dyanmics}
  \begin{bmatrix}
\dot{y} \\
\dot{z} \\
\dot{\phi} \\
\dot{v}_y\\
\dot{v}_z \\
{\ddot{\phi}}
\end{bmatrix}
=
{\begin{bmatrix}
v_y \\
v_z \\
\dot{\phi} \\
0 \\
-g \\
0
\end{bmatrix}}
+
\begin{bmatrix}
0 & 0 \\
0 & 0 \\
0 & 0 \\
\frac{1}{m/\theta} & 0 \\
0 & \frac{1}{m/\theta } \\
-\frac{l\cos{\phi}}{J/\theta} & \frac{l\sin{\phi}}{J/\theta}
\end{bmatrix}
\begin{bmatrix}
F_y \\
F_z
\end{bmatrix},  
\end{equation}
where $y$, $z$ and $\phi$ are the position in $y$ axis,  position in $z$ axis and the angle between the rocket body and upwards direction,  $m=1$, $l=1$, and $J=\frac{m}{12}(2l)^2$ denotes the mass, the length from center of mass (COM) to the thruster, and inertia of the rocket. Here we assume that the COM is in the middle of the rocket and is time-invariant to simplify the modeling. $g=9.81$ is the gravitational constant, and  $F_y$ and $F_z$ denotes the thrust forces in the $y$ and $z$ axes, respectively. The parameter $\theta$ reflects the mass and inertia change due to fuel consumption, and is set to be $\theta(t)\in\Theta=[1,5]$, $\dot\theta(t)\in[0,0.1]$ for all $t\geq 0$. Note that $\theta=5$ corresponds to the mass and inertia reduced to 20\% of those associated with $\theta=1$.  

\begin{figure}[ht]
    \centering
    \begin{subfigure}[b]{0.4\columnwidth}
        \centering
\includegraphics[width= 1\columnwidth]{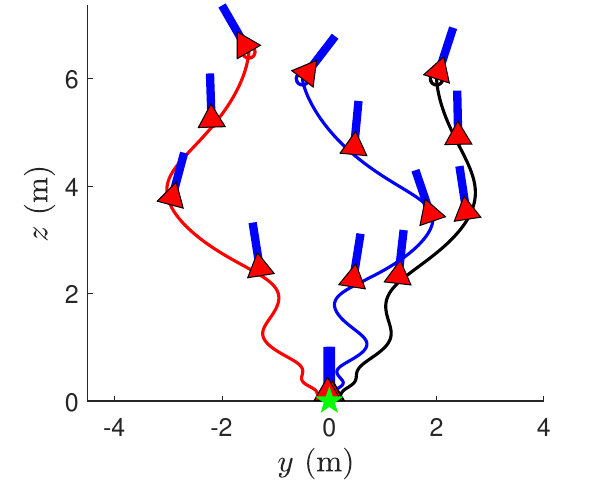}\\ 
\includegraphics[width= 1\columnwidth]{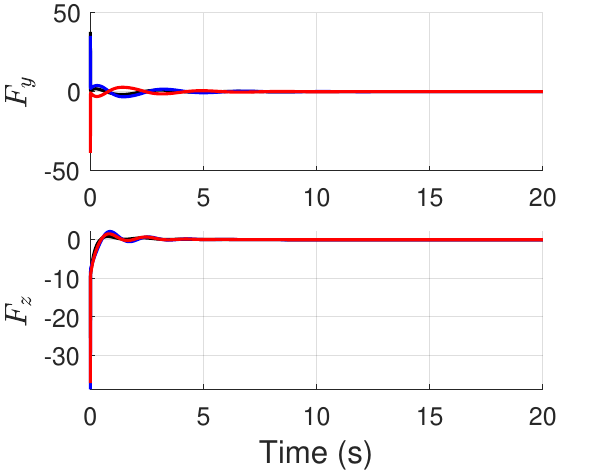}
        \caption{Pre-defined control law \cref{eq:control-law-Y-X}}
        \label{fig:rocket-predefine-u}
    \end{subfigure}
      \begin{subfigure}[b]{0.4\columnwidth}
    \includegraphics[width=1\columnwidth]{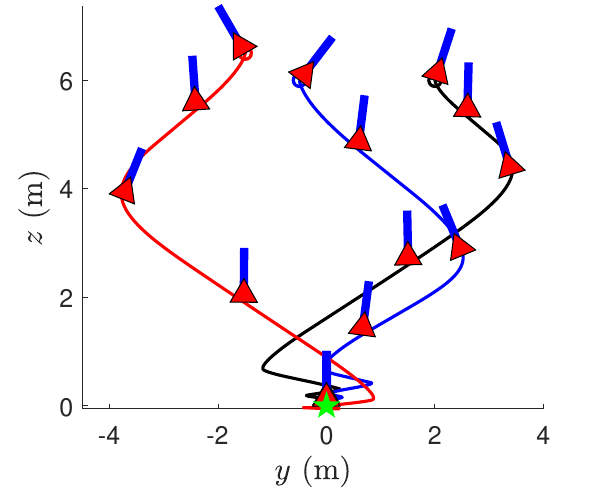} \\ 
\includegraphics[width= 1\columnwidth]{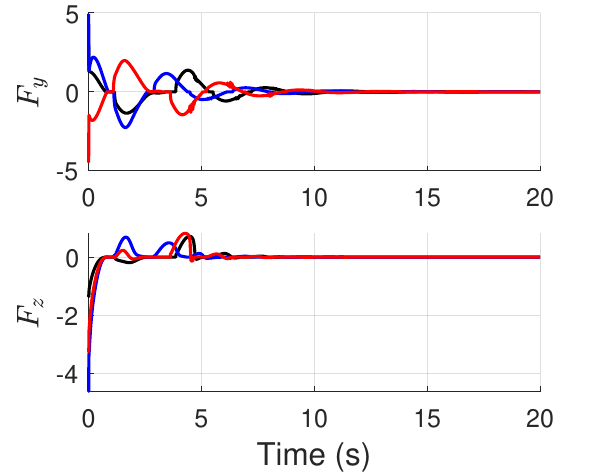}
        \caption{Min-norm control law \cref{eq:min-norm-formulation}}
        \label{fig:rocket-minnorm-u}
    \end{subfigure}
    \caption{Trajectories of states (top) and control inputs (bottom) starting from four initial points under the pre-defined control law \cref{eq:control-law-Y-X} (left) and the min-norm control law \cref{eq:min-norm-formulation} (right).}
    \label{fig:rocket-traj}
    \vspace{-6mm}
\end{figure}

To remove the constant $g$ so that we could rewrite the dynamics in the linear-like form \cref{eq:dynamics-linear-form}, we define $F_z = \hat F_z + \frac{mg}{\theta}$, and rewrite the dynamics in the form of new control input $\hat u = [F_y, \hat F_z]^T$. With Taylor series expansion to approximate the trigonometric functions, the new dynamics can further be rewritten in the linear-like polynomial form \cref{eq:dynamics-linear-form}. In solving the SOS problem \cref{eq:opt-joint-syn-sos}, we selected the Lyapunov and controller matrix variables as $X(\phi,\theta)$ and $Y(\phi,\theta)$, both of which have the order of 2. We decided not to include more dependent states in $X$ and $Y$ to reduce the computational load. The local synthesis set $\mcX$ is selected to be $\mcX = \left\{x\in\mbR^6: y^2+z^2\leq 36,~\phi^2+(\dot\phi)^2\leq (\frac{\pi}{3})^2,~v_y^2+v_z^2\leq 4\right\}$. 

After getting the PD-CLF and the PD controller from solving \cref{eq:opt-joint-syn-sos}, we simulated the system starting from initial states $x(0)=[y_0,z_0, \phi_0, 0, -0.5, 0]^T$ under the trajectory $\theta(t) =1+0.08t$. The state and input trajectories for three cases of $(y_0,z_0, \phi_0)$ under the predefined control law and min-norm control law, are shown in \cref{fig:rocket-traj}. Notice that $x(0)$ for all three cases is actually out of the set $\mcX$ used in the synthesis procedure. However, both control laws were able to land the rocket successfully in all these cases. Additionally, the magnitude of the min-norm control signal is significantly smaller compared to that of the predefined control law. This reduction offers notable advantages in certain scenarios, particularly in avoiding input saturation.

\section{Conclusions}\label{sec:conclusion}
This paper presented an approach to gain-scheduled stabilization of nonlinear parameter-varying (NPV) systems through parameter-dependent (PD) control Lyapunov functions (CLFs). For control-affine polynomial NPV systems, it introduces convex conditions based on sum of squares (SOS) programming to 
jointly synthesize a 
PD-CLF  and a stabilizing controller. Simulation results validated the effectiveness of the proposed approach, demonstrating its potential for enhancing the stabilization of NPV systems. The proposed PD-CLF provides a way to control nonlinear systems with large operating envelopes. Compared to the existing LPV approach, it does not need linearization or over-approximation that is typically required to obtain an LPV system. 

Future work includes 
consideration of potential errors in measuring and estimating the scheduling parameters that have been studied in the LPV setting \cite{Sato13Inexact,Zhao17Inexact}.
\bibliographystyle{ieeetr}
\bibliography{bib/refs-pan,bib/refs-new}

\end{document}